\documentclass[12pt]{article}
\usepackage{mathtools,amsthm,amsmath,amsfonts,amssymb,tikz-cd,enumitem, graphicx,mathrsfs,bbm}

\tikzcdset{scale cd/.style={every label/.append style={scale=#1},
    cells={nodes={scale=#1}}}}
\usepackage{url}
\usepackage{esvect}
\makeatletter
\newcommand{\address}[1]{\gdef\@address{#1}}
\newcommand{\email}[1]{\gdef\@email{\url{#1}}}
\newcommand{\@endstuff}{\par\vspace{\baselineskip}\noindent\small
\begin{tabular}{@{}l}\@address\\\textit{E-mail address:} \@email\end{tabular}}
\AtEndDocument{\@endstuff}
\makeatother
\usepackage[utf8]{inputenc}
\counterwithin{equation}{section}
\usepackage[pdfencoding=auto,psdextra,colorlinks=true,linkcolor=blue,citecolor=blue,urlcolor = orange]{hyperref}
\allowdisplaybreaks
\usepackage[margin = 1.25in]{geometry}
\usepackage{setspace}
\newtheorem{theorem}{Theorem}[section]
\newtheorem{definition}[theorem]{Definition}
\newtheorem{example}[theorem]{Example}
\newtheorem{lemma}[theorem]{Lemma}
\newtheorem{question}[theorem]{Question}
\newtheorem{remark}[theorem]{Remark}
\newtheorem{proposition}[theorem]{Proposition}
\newtheorem{corollary}[theorem]{Corollary}
\newtheorem{assumption}[theorem]{Assumption}
\newcommand{\citep}[1]{\cite{#1}}

\newcommand{\dvol}{\text{\normalfont dvol}}

\newcommand{\supp}{\text{\normalfont supp}}

\newcommand{\ind}{\text{\normalfont ind}}

\newcommand{\even}{\text{\normalfont{even}}}
\newcommand{\odd}{\text{\normalfont{odd}}}

\newcommand{\cptwo}{\mathbb{C}\rm{P}^2}
\newcommand{\professor}{\text{Prof.\ }\hspace{-0.03125mm}}
\newcommand{\doctor}{\text{Dr.\ }\hspace{-0.03125mm}}

\begin{document}
\title{\textbf{Symplectic semi-characteristics}}
\author{Hao Zhuang}
\date{\today}
\maketitle
\address{Beijing International Center for Mathematical Research, Peking University}
\email{hzhuang@pku.edu.cn}
\begin{abstract}
We study the symplectic semi-characteristic of a closed $4n$-dimensional symplectic manifold. First, using the even-degree part of the primitive cohomology, we define the symplectic semi-characteristic. Second, using a vector field with nondegenerate zero points, we prove a counting formula for the symplectic semi-characteristic. As corollaries of the counting formula, we obtain a vanishing property and the fact that the definition of the symplectic semi-characteristic is independent of the choices of symplectic forms. 
\end{abstract}

\tableofcontents
\section{Introduction}
In three consecutive papers \cite[(3.14), (3.22)]{tty1st}, \cite[(1.5), (1.6)]{tty2nd}, and \cite[(1.2), Theorem 3.1]{tty3rd}, Tsai, Tseng, and Yau introduced the $p$-filtered cohomology groups $F^pH^k(M,\omega)\ (0 \leqslant k\leqslant 1+2p+\dim M)$  of any symplectic manifold $(M,\omega)$. The $p = 0$ case in \cite[(3.14), (3.22)]{tty1st} and \cite[(1.5), (1.6)]{tty2nd} is called the primitive cohomology of $(M,\omega)$, and the $p\geqslant 1$ case in \cite[(1.2), Theorem 3.1]{tty3rd} is generalized from the primitive cohomology. Different from the classical de Rham cohomology, this $p$-filtered cohomology includes the information of the symplectic form $\omega$. Thus, an important application of the $p$-filtered cohomology is to distinguish different symplectic structures. For the examples of this application, see \cite[Section 4]{tty2nd} (by computing the primitive cohomology groups) and \cite[Section 6]{tty3rd} (by computing the product structures). In addition, Tanaka and Tseng \cite[Theorem 1.1]{tanaka_tseng_2018} proved that the mapping cone complex determined by the map $\wedge\omega^{p+1}$ between de Rham complexes computes the $p$-filtered cohomology.

In this paper, we focus on the $p = 0$ part, the primitive cohomology. Our work starts from an interesting fact about the primitive cohomology: When the symplectic manifold $(M,\omega)$ is closed, the Euler characteristic 
\begin{align}\label{primitive euler}
    \dim F^0H^{\text{even}}(M,\omega) - \dim F^0H^{\text{odd}}(M,\omega)
\end{align}
of the primitive cohomology is equal to zero. This fact was originally proved in \cite[Proposition 3.26]{tty1st} and \cite[Proposition 3.7]{tty2nd} by verifying the duality between cohomology groups. Recently, using Tanaka and Tseng's mapping cone complex, Clausen, Tang, and Tseng proved the symplectic Morse inequality \cite[Theorem 1.4]{tangtsengclausensymplecticwitten}, also showing (\ref{primitive euler}) is zero. Intuitively speaking, we may say that the even-degree part of the primitive cohomology contains the same amount of information as the odd-degree part. 

We then ask whether $F^0H^{\text{even}}(M,\omega)$ is an obstruction to some geometric object. Here, to some extent, we are motivated by the Kervaire semi-characteristic in a similar scenario. Recall that for any odd-dimensional closed oriented manifold $N$, the Euler characteristic of the de Rham cohomology of $N$ is zero. Let $b_k$ be the dimension of the $k$-th de Rham cohomology group of $N$. The $\mathbb{Z}_2$-valued Kervaire semi-characteristic (See \cite[Introduction]{kervaireoriginalpaper}, \cite[Section 1]{lusztig}, and \cite[Section 4]{atiyah2013vector}) of $N$ is defined to be
\begin{align}\label{classical kervaire semi definition}
    \sum_{k\text{\ is even}}b_k \mod 2.
\end{align}
When $\dim N = 4n+1$, the $\mathbb{Z}_2$-valued Kervaire semi-characteristic satisfies Atiyah's vanishing theorem \cite[Theorem 4.1]{atiyah2013vector} and Zhang's counting formula \cite[Theorem 1.3]{zhangcountingmod2indexkervairesemi}. 
Both the vanishing theorem and the counting formula involve two vector fields, showing that the $\mathbb{Z}_2$-valued Kervaire semi-characteristic is an obstruction to a certain pair of vector fields. Now, back to our symplectic situation, we precisely state our main question: 
\begin{question}\label{main question}
\normalfont
    For any closed symplectic manifold $(M,\omega)$, which geometric object on $M$ does the even-degree part of the primitive cohomology of $(M,\omega)$ obstruct? Is the geometric object a pair of vector fields or something else?
\end{question}

We give an answer to Question \ref{main question} for $4n$-dimensional closed symplectic manifolds. 

\begin{assumption}\label{only assumption}
\normalfont
    Throughout this paper, when there is no particular clarification, $(M,\omega)$ means a $4n$-dimensional closed symplectic manifold $M$ equipped with a symplectic form $\omega$. 
\end{assumption}

We recall Tanaka and Tseng's mapping cone complex $(C^*(M,\omega), \partial_C)$ that computes the primitive cohomology of $(M,\omega)$. Let $\Omega^k(M)$ be the space of smooth $k$-forms on $M$. According to \cite[Section 3.1]{tanaka_tseng_2018} and \cite[Definition 1.1]{tangtsengclausensymplecticwitten}, the space of $k$-cochains is $$C^k(M,\omega) \coloneqq \Omega^k(M)\oplus\Omega^{k-1}(M)\ \  (k = 0,1,\cdots,4n+1).$$ Let $d$ be the de Rham exterior differentiation. The boundary map is 
\begin{align*}
\begin{split}
    \partial_C: C^k(M,\omega)&\to C^{k+1}(M,\omega)\\
   \begin{bmatrix}
       \alpha\\
       \beta
   \end{bmatrix}&\mapsto \begin{bmatrix}
    d & \omega\\
    0 & -d
\end{bmatrix}\begin{bmatrix}
    \alpha\\
    \beta
\end{bmatrix} = \begin{bmatrix}
    d\alpha+\omega\wedge\beta\\
    -d\beta
\end{bmatrix}.
\end{split}
\end{align*}
Here, we write the pair $(\alpha,\beta)\in\Omega^k(M)\oplus\Omega^{k-1}(M)$ as a column for the convenience of using matrices and operators later.

The $k$-th cohomology group of $(C^*(M,\omega),\partial_C)$ is exactly $F^0H^k(M,\omega)$. Let $b^\omega_{k}$ be the dimension of $F^0H^k(M,\omega)$. We define the symplectic semi-characteristic of $(M,\omega)$ as follows. 

\begin{definition}\label{definition of symp semi char}
\normalfont
    We call the $\mathbb{Z}_2$-valued number
    \begin{align}\label{symplectic semi definition we need it}
        \kappa(M,\omega)\coloneqq \sum_{k\text{\ is even}}b_{k}^\omega \mod 2
    \end{align}
    the symplectic semi-characteristic of $(M,\omega)$. 
\end{definition}

    To state our main result, we review the definition of nondegenerate vector fields. Let $V$ be a smooth vector field on $M$. Then, following \cite[Section 1.6]{bgv}, at each zero point $p$ of $V$, we define a homomorphism
\begin{align*}
    \Phi_p: T_pM&\to T_pM\\
    v&\mapsto [V,\tilde{v}](p),
\end{align*}
where $\tilde{v}$ is a vector field extending the tangent vector $v$, and $[\cdot,\cdot]$ is the Lie bracket between vector fields. This $\Phi_p$ is independent of the extension $\tilde{v}$ since $V$ equals zero at $p$. 

\begin{definition}\normalfont
    A smooth vector field $V$ on $M$ is called nondegenerate if either $V$ is nonvanishing, or $\Phi_p$ is invertible for each zero point $p$ of $V$.  
\end{definition}

Such a nondegenerate vector field always exists because by \cite[Theorem 6.6]{milnor1963morse}, there is always a Morse function on $M$. 
Now, our main result is as follows (cf. \cite[Theorem 1.3]{zhangcountingmod2indexkervairesemi}). 

\begin{theorem}\label{counting formula theorem}
    Let $V$ be a smooth nondegenerate vector field on $M$. Then, 
    \begin{align}\label{main result main formula}
        \kappa(M,\omega) = \text{the number of zero points of\ } V \mod 2
    \end{align}
    is the counting formula for the symplectic semi-characteristic under Assumption \ref{only assumption}.
\end{theorem}

\begin{remark}\normalfont
    By the Poincar\'e-Hopf index formula \cite[Theorem 4.5]{wittendeformationweipingzhang}, formula (\ref{main result main formula}) is also mod 2 to the Euler characteristic of the de Rham cohomology of $M$. 
\end{remark}

\begin{remark}
\normalfont
    A special situation is when the de Rham cohomology class of $\omega$ is integral. In this situation, by \cite[Theorem 7.1]{tanaka_tseng_2018}, Theorem \ref{counting formula theorem} computes the classical Kervaire semi-characteristic (\ref{classical kervaire semi definition}) of the circle bundle over $M$ induced by the line bundle associated with $\omega$. 
\end{remark}
The main idea of the proof is, we find a skew-adjoint operator like Zhang's construction \cite[(1.1)]{zhangcountingmod2indexkervairesemi}. Then, we show that $\kappa(M,\omega)$ is equal to the Atiyah-Singer mod 2 index (See Definition \ref{atiyah-singer mod 2 index def}) of this operator. Afterwards, like \cite[(2.1)]{zhangcountingmod2indexkervairesemi}, we apply the Witten deformation and the Bismut-Lebeau asymptotic analysis (see \cite[Section 2]{witten}, \cite[Chapter VIII-X]{bismutandlebeau}, and \cite[Chapters 4-7]{wittendeformationweipingzhang}) to the operator, compute its mod 2 index, and then obtain Theorem \ref{counting formula theorem}. 

A corollary of Theorem \ref{counting formula theorem} is an Atiyah type vanishing property (cf. \cite[Theorem 4.1]{atiyah2013vector}): 
\begin{corollary}\label{vanishing theorem}
    The semi-characteristic $\kappa(M,\omega) = 0$ when there is a nonvanishing smooth vector field on $M$.
\end{corollary}
Another way to prove Corollary \ref{vanishing theorem} without using Theorem \ref{counting formula theorem} is in Remark \ref{atiyah trick}. 

By Example \ref{weak replacement of the euler char}, the opposite direction of Corollary \ref{vanishing theorem} is not true. This is different from the Euler characteristic of the de Rham cohomology of $M$. 

In addition, although we use the symplectic form $\omega$ to define $\kappa(M,\omega)$, the nondegenerate vector field always exists and is independent of the symplectic structure. Thus, we have: 
\begin{corollary}\label{independence result}
   The definition of  $\kappa(M,\omega)$ is independent of the chosen symplectic form. 
\end{corollary}

\begin{remark}\normalfont
    The definition of the symplectic semi-characteristic can be assigned to any closed symplectic manifold without assuming that $\dim M = 4n$. However, by Example \ref{counter example 2}, the counting formula does not work when the dimension is $4n+2$. Thus, the $(4n+2)$-dimensional part of Question \ref{main question} is still open. 
\end{remark}

This paper is organized as follows. In Section \ref{clifford actions and operators section}, we review Clifford actions and find a skew-adjoint operator so that $\kappa(M,\omega)$ equals the Atiyah-Singer mod 2 index of this operator. In Section \ref{Symplectic Witten deformation section}, we carry out necessary analytic details about this operator. In Section \ref{Counting formula section}, we prove Theorem \ref{counting formula theorem} based on these analytic details. In Section \ref{Examples section}, we give some examples and further discussions. 

\vspace{+3mm}
\noindent\textbf{Acknowledgments}. I want to thank my PhD supervisor \professor Xiang Tang for introducing me the analytic method in the symplectic Morse theory and offering practical suggestions to simplify several steps. Also, I want to thank \professor Li-Sheng Tseng for mentioning the relations between mapping cone complexes and odd sphere bundles, and \doctor David Clausen for his talk on the symplectic Morse inequality. Meanwhile, I want to thank \professor Yi Lin for pointing out the influences of the K\"ahler condition, and \professor Christopher Seaton for asking about the relations between the classical Euler characteristic and the symplectic semi-characteristic. Furthermore, I want to thank the anonymous referees for the helpful comments and suggestions on mathematics and writing. Finally, I want to thank Washington University in St. Louis for supporting my PhD study during the completion of this project. 

\section{Clifford actions and operators}\label{clifford actions and operators section}
In this section, we clarify technical details about the Clifford actions of tangent vectors. Also, we give the skew-adjoint operator that we will work with. 

After choosing an almost complex structure $J$ on $M$, we let $g$ be the Riemannian metric
$$g(\cdot,\cdot) = \omega(\cdot,J\cdot)$$ 
on $M$. Then, we equip $M$ with the orientation $\omega\wedge\cdots\wedge\omega$ and let $\star$ be the Hodge star operator. Let $\dvol = \star 1$ be the volume form of $M$, we define the $L^2$-norm (also, the inner product)
\begin{align}\label{L2 on smooth}
    \|\alpha\| = \left(\int_M g(\alpha,\alpha)\dvol\right)^{1/2}
\end{align}
on $\Omega^k(M)$. 
We require $\Omega^k(M)\perp \Omega^{\ell}(M)$ when $k\neq\ell$. For the pairs of forms, following \cite[(2.2)]{tangtsengclausensymplecticwitten}, on $C^k(M,\omega) = \Omega^k(M)\oplus\Omega^{k-1}(M)$, we define 
\begin{align*}
    g\left(\begin{bmatrix}
        \alpha\\
        \beta
    \end{bmatrix},\begin{bmatrix}
        \alpha'\\
        \beta'
    \end{bmatrix}\right) = g(\alpha,\alpha')+g(\beta,\beta'). 
\end{align*}
Like (\ref{L2 on smooth}), we have the $L^2$-norm (also, the inner product)
\begin{align}\label{L2 on pair}
    \left\|\begin{bmatrix}
    \alpha\\
    \beta
\end{bmatrix}\right\| = \left(\int_M g(\alpha,\alpha)\dvol+\int_M g(\beta,\beta)\dvol\right)^{1/2}
\end{align}
on $C^k(M,\omega)$. We require $C^k(M,\omega)\perp C^{\ell}(M,\omega)$ when $k\neq\ell$. 
In addition, we let $d^*$ be the formal adjoint of $d$ with respect to the inner product induced by (\ref{L2 on smooth}), and $$\omega^*\lrcorner: \Omega^k(M)\to\Omega^{k-2}(M)$$ be the adjoint of 
\begin{align*}
\omega\wedge: \Omega^k(M)&\to\Omega^{k+2}(M)\\
    \alpha&\mapsto \omega\wedge\alpha
\end{align*}
with respect to the same inner product. For convenience, we will omit the ``$\lrcorner$'' behind $\omega^*$ and the ``$\wedge$'' behind $\omega$ when there is no ambiguity. Recall the mapping cone complex 
\begin{align}\label{the mapping cone complex by partial c}
\begin{split}
    \partial_C: \Omega^k(M)\oplus\Omega^{k-1}(M)&\to\Omega^{k+1}(M)\oplus\Omega^k(M)\\
    (\alpha,\beta)&\mapsto \begin{bmatrix}
        d & \omega\\
        0 & -d
    \end{bmatrix}\begin{bmatrix}
        \alpha\\
        \beta
    \end{bmatrix}. 
\end{split}
\end{align}
The formal adjoint of $\partial_C$ is 
$$\partial_C^* = \begin{bmatrix}
    d^* & 0\\
    \omega^* & -d^*
\end{bmatrix}$$
with respect to the inner product induced by (\ref{L2 on pair}).

\begin{proposition}\label{proposition hodge theorem}
    The kernel of the Dirac type operator 
$\partial_C+\partial_C^*$
(also, the kernel of the Laplacian $(\partial_C+\partial_C^*)^2$) is isomorphic to the primitive cohomology of $(M,\omega)$. In particular, 
$$\ker\left((\partial_C+\partial_C^*)^2: C^k(M,\omega)\to C^k(M,\omega)\right)$$
is isomorphic to the $k$-th primitive cohomology group. 
\end{proposition}
\begin{proof}
    We can check that (\ref{the mapping cone complex by partial c}) defines an elliptic complex \cite[Definition 10.4.28]{nicolaescu2020lectures}. By the properties \cite[Theorem 10.4.30]{nicolaescu2020lectures} of an elliptic complex, the complex defined by (\ref{the mapping cone complex by partial c}) satisfies the Hodge decomposition theorem. Thus, the kernel of $(\partial_C+\partial_C^*)^2|_{C^k(M,\omega)}$ is isomorphic to the $k$-th primitive cohomology group. 
\end{proof}

For any (globally or locally defined) vector field $Y$ on $M$, we have two Clifford actions $$\hat{c}(Y) = Y^*\wedge+Y\lrcorner \text{\ \ and\ \ } c(Y) = Y^*\wedge - Y\lrcorner.$$ 
Given any oriented local orthonormal frame $e_1,\cdots,e_{4n}$ of $TM$, the Clifford action of the volume form $\dvol$ is expressed as  
$$\hat{c}(\dvol) = \hat{c}(e_1)\cdots\hat{c}(e_{4n}).$$ 
This $\hat{c}(\dvol)$ is independent of the choices of oriented local orthonormal frames. Following \cite[Section 3]{atiyah2013vector}, we verify some interactions between the Hodge star, Clifford actions, and differential forms. Recall that the dimension of $M$ is $4n$. 
\begin{lemma}\label{lemma clifford 1}
    For all $\alpha\in\Omega^k(M)$, we have 
    $\hat{c}(\dvol)\alpha = (-1)^{k(k+1)/2}\star\alpha$. 
\end{lemma}
\begin{proof}
  Let $e_1,\cdots,e_{4n}$ be an oriented local orthonormal frame. Suppose $\alpha = e_{i_1}^*\wedge\cdots\wedge e_{i_k}^*$ such that 
    $$e_{i_1}^*\wedge\cdots\wedge e_{i_k}^*\wedge e_{j_1}^*\wedge\cdots\wedge e_{j_{4n-k}}^* = e_1^*\wedge\cdots\wedge e_{4n}^*. $$
    Then, we have 
    \begin{align*}
         \hat{c}(\dvol)\alpha 
        =\ & \hat{c}(e_{i_1})\hat{c}(e_{i_2})\cdots\hat{c}(e_{i_k})\hat{c}(e_{j_1})\cdots\hat{c}(e_{j_{4n-k}})\alpha\\
        =\ & (-1)^{k(4n-k)}\hat{c}(e_{j_1})\cdots\hat{c}(e_{j_{4n-1}})\hat{c}(e_{i_1})\hat{c}(e_{i_2})\cdots\hat{c}(e_{i_k})\alpha\\
        =\ & (-1)^{k(4n-k)+\frac{(0+k-1)k}{2}}\star(e_{i_1}^*\wedge\cdots\wedge e_{i_k}^*)\\
        =\ & (-1)^{k(k+1)/2}\star\alpha.
    \end{align*}
    The general case of $\alpha$ is straightforward. 
\end{proof}

\begin{lemma}\label{lemma clifford 2}
    For all $\alpha\in\Omega^k(M)$, we have $\hat{c}(\dvol)(\omega^*\lrcorner\alpha) = -\omega\wedge\hat{c}(\dvol)\alpha$. 
\end{lemma}
\begin{proof}
    Using $\omega^*\lrcorner\alpha = (-1)^k\star\omega\star\alpha$ (See \cite[Section 2.1]{tangtsengclausensymplecticwitten}), we find
    \begin{align*}
        \hat{c}(\dvol)(\omega^*\lrcorner\alpha)
        =\ & \hat{c}(\dvol)\left((-1)^k\star\omega\star\alpha\right)\\
        =\ & (-1)^{\frac{(k-2+1)(k-2)}{2}}\star\left((-1)^k\star\omega\star\alpha\right) \ \text{(by Lemma \ref{lemma clifford 1})}\\
        =\ & (-1)^{\frac{(k-2+1)(k-2)}{2}}\star\left((-1)^k\star\left(\omega\wedge (-1)^{\frac{k(k+1)}{2}}\hat{c}(\dvol)\alpha\right)\right)\\
        =\ & (-1)^{k^2+1}\star\star(\omega\wedge(\hat{c}(\dvol)\alpha)).
    \end{align*}
    When $k$ is odd, $\omega\wedge\hat{c}(\dvol)\alpha$ is an odd-degree form, making $\star\star = -1$ and then $(-1)^{k^2+1}\star\star = -1$. 
    Similarly, when $k$ is even, we have $\star\star = 1$ and then $(-1)^{k^2+1}\star\star = -1$. 
    Thus, we obtain 
    $\hat{c}(\dvol)(\omega^*\lrcorner\alpha) = -\omega\wedge(\hat{c}(\dvol)\alpha)$. 
\end{proof}

We denote $\bigoplus_{k = 0}^{2n}\Omega^{2k}(M)$ (resp. $\bigoplus_{k = 1}^{2n}\Omega^{2k-1}(M)$) by $\Omega^\even(M)$ (resp. $\Omega^\odd(M)$). Using Lemma \ref{lemma clifford 1} and Lemma \ref{lemma clifford 2}, we obtain a skew-adjoint operator as follows. 
\begin{proposition}\label{proposition skew-adjoint signature operator}
    The operator 
    \begin{align}\label{skew-adj operator first}
    \begin{bmatrix}
        0 & 1\\
        1 & 0
    \end{bmatrix}\begin{bmatrix}
        \hat{c}(\dvol) & \\
        & \hat{c}(\dvol)
    \end{bmatrix}\begin{bmatrix}
        d+d^* & \omega\\
        \omega^* & -d-d^*
    \end{bmatrix}
    \end{align}
    on $\Omega^{\even}(M)\oplus\Omega^\odd(M)$ is skew-adjoint. 
\end{proposition}
\begin{proof}
    Using Lemma \ref{lemma clifford 1} and \cite[Definition 6.1(2)]{warner2013foundations}, for all $\alpha\in\Omega^k(M)$, 
    \begin{align*}
         \hat{c}(\dvol)(d+d^*)\alpha =\ & (-1)^{\frac{(k+1)(k+2)}{2}}\star d\alpha + (-1)^{\frac{(k-1)k}{2}}\star (-1) \star d\star\alpha\\
        =\ & (-1)^{\frac{(k+1)(k+2)}{2}}\star d\alpha \\
        & + (-1)^{\frac{(k-1)k}{2}+1}(-1)^{(4n-k+1)(4n-4n+k-1)}d\star\alpha\\
        =\ & (-1)^{\frac{(k+1)(k+2)}{2}}\star d\alpha + (-1)^{\frac{k(k+1)}{2}}d\star\alpha. 
    \end{align*}
    Similarly, 
    \begin{align*}
         (d+d^*)\hat{c}(\dvol)\alpha 
        =\ & d (-1)^{\frac{k(k+1)}{2}}\star\alpha + d^* (-1)^{\frac{k(k+1)}{2}}\star\alpha \\
        =\ & (-1)^{\frac{k(k+1)}{2}}d\star\alpha + (-1)\star d\star (-1)^{\frac{k(k+1)}{2}}\star\alpha \\
        =\ & (-1)^{\frac{k(k+1)}{2}}d\star\alpha + (-1)^{\frac{(k+1)(k+2)}{2}}\star d\alpha. 
    \end{align*}
    Thus, we have $(d+d^*)\hat{c}(\dvol) = \hat{c}(\dvol)(d+d^*)$. 
    
    Now, 
   since $\hat{c}(\dvol)^* = \hat{c}(e_{4n})\cdots\hat{c}(e_1) = \hat{c}(e_1)\cdots\hat{c}(e_{4n}) = \hat{c}(\dvol)$, we find 
    \begin{align*}
        &\left(\begin{bmatrix}
            0 & 1\\
            1 & 0
        \end{bmatrix}\begin{bmatrix}
            \hat{c}(\dvol) & \\
            & \hat{c}(\dvol)
        \end{bmatrix}\begin{bmatrix}
            d+d^* & \omega \\
            \omega^* & -d-d^*
        \end{bmatrix}\right)^*\\
        =\ & \left(\begin{bmatrix}
            \hat{c}(\dvol)\omega^* & -\hat{c}(\dvol)(d+d^*)\\
            \hat{c}(\dvol)(d+d^*) & \hat{c}(\dvol)\omega
        \end{bmatrix}\right)^*\\
        =\ &\begin{bmatrix}
            \omega\hat{c}(\dvol) & (d+d^*)\hat{c}(\dvol)\\
            -(d+d^*)\hat{c}(\dvol) & \omega^*\hat{c}(\dvol)
        \end{bmatrix}\\
        =\ & -\begin{bmatrix}
            \hat{c}(\dvol)\omega^* & -\hat{c}(\dvol)(d+d^*)\\
            \hat{c}(\dvol)(d+d^*) & \hat{c}(\dvol)\omega
        \end{bmatrix} \text{\ (by Lemma \ref{lemma clifford 2})}\\
        =\ & -\begin{bmatrix}
            0 & 1\\
            1 & 0
        \end{bmatrix}\begin{bmatrix}
            \hat{c}(\dvol) & \\
            & \hat{c}(\dvol)
        \end{bmatrix}\begin{bmatrix}
            d+d^* & \omega \\
            \omega^* & -d-d^*
        \end{bmatrix}.
    \end{align*}
    Thus, the operator (\ref{skew-adj operator first}) is skew-adjoint. 
\end{proof}

We recall the definition of the Atiyah-Singer mod 2 index. In \cite[Theorem A]{atiyahsingerskewadjoint}, the mod 2 index was defined for real Fredholm skew-adjoint operators. However, by functional calculus \cite[Definition 1.13]{higson2000analytic}, we state the version \cite[(7.5)]{wittendeformationweipingzhang} for real elliptic skew-adjoint operators: 
\begin{definition}\label{atiyah-singer mod 2 index def}\normalfont
    Given a real elliptic skew-adjoint operator $D$, its Atiyah-Singer mod 2 index is the $\mathbb{Z}_2$-valued number 
    $$\dim\ker D\mod 2$$
    and is denoted by $\ind_2 D$. 
\end{definition}

According to the definition of $\kappa(M,\omega)$ and the identification between kernels and cohomology groups, we have: 
\begin{corollary}\label{corollary signature = semi char}
    The Atiyah-Singer mod 2 index of 
    \begin{align}\label{very first skew-adjoint op}
        \begin{bmatrix}
        0 & 1\\
        1 & 0
    \end{bmatrix}\begin{bmatrix}
        \hat{c}(\dvol) & \\
        & \hat{c}(\dvol)
    \end{bmatrix}\begin{bmatrix}
        d+d^* & \omega\\
        \omega^* & -d-d^*
    \end{bmatrix}
    \end{align}
    on $\Omega^{\even}(M)\oplus\Omega^\odd(M)$
    is equal to $\kappa(M,\omega)$. 
\end{corollary}
\begin{proof}
    This is verified by Definition \ref{definition of symp semi char}, Proposition \ref{proposition hodge theorem}, and Proposition \ref{proposition skew-adjoint signature operator}. With $\begin{bmatrix}
        0 & 1\\
        1 & 0
    \end{bmatrix}$, the operator (\ref{very first skew-adjoint op}) preserves the parity of the grading, mapping $\Omega^{\even}(M)\oplus\Omega^\odd(M)$ (resp. $\Omega^{\odd}(M)\oplus\Omega^\even(M)$) into $\Omega^{\even}(M)\oplus\Omega^\odd(M)$ (resp. $\Omega^{\odd}(M)\oplus\Omega^\even(M)$). When we restrict (\ref{very first skew-adjoint op}) to $\Omega^{\even}(M)\oplus\Omega^\odd(M)$, its kernel counts half of the $b^\omega_k$\hspace{+0.5mm}'s. 
\end{proof}

Similar to the Fredholm index \cite[Theorem 3.11]{conway2007course}, as stated in \cite[Proposition 5.1]{atiyahsingerskewadjoint} and \cite[Section 2]{atiyah_singer_1971index_V}, the mod 2 index of a real skew-adjoint elliptic operator on a compact manifold is homotopy invariant: 
\begin{proposition}\label{continuous family}
   Given a real skew-adjoint elliptic operator $D$ on $M$, the $\ind_2 D$ is invariant under the continuous deformation of $D$.   
\end{proposition}
\begin{remark}
\normalfont
     The invariance of $\ind_2$ in \cite[Proposition 5.1]{atiyahsingerskewadjoint} is for bounded real skew-adjoint Fredholm operators. It passes to the elliptic operators on $M$ in this way \cite[Section 4]{atiyah2013vector}: For $D$ in Proposition \ref{continuous family}, $1+(-D^2)$ is self-adjoint and positive. We have the compact operator
     \begin{align*}
         \left(1+(-D^2)\right)^{-1/2}
     \end{align*}
     defined by functional calculus \cite[Definition 1.13]{higson2000analytic}. Then,  
     \begin{align*}
         D\circ\left(1+(-D^2)\right)^{-1/2}
     \end{align*}
     is a bounded real skew-adjoint Fredholm operator. 
\end{remark} 

The next proposition gives us the skew-adjoint operator similar to \cite[(1.1)]{zhangcountingmod2indexkervairesemi}. 
\begin{proposition}
    The Atiyah-Singer mod 2 index of the skew-adjoint operator 
    \begin{align}\label{before witten def}
        \begin{bmatrix}
        \dfrac{1}{2}(\omega^*-\omega) & -d-d^*\\
        d+d^* & \dfrac{1}{2}(\omega-\omega^*)
    \end{bmatrix}
    \end{align}
    on $\Omega^\even(M)\oplus\Omega^\odd(M)$ is equal to $\kappa(M,\omega)$. 
\end{proposition}
\begin{proof}
By Lemma \ref{lemma clifford 2}, the operator
$$\dfrac{1}{2}\begin{bmatrix}
        \hat{c}(\dvol)(\omega^*+\omega) & \\
        & \hat{c}(\dvol)(\omega+\omega^*)
    \end{bmatrix}$$
    is skew-adjoint on 
 $\Omega^\even(M)\oplus\Omega^\odd(M)$. Then, by Corollary \ref{corollary signature = semi char}, we find 
   \begin{align*}
    & \kappa(M,\omega)\\
       =\ &\ind_2\begin{bmatrix}
        0 & 1\\
        1 & 0
    \end{bmatrix}\begin{bmatrix}
        \hat{c}(\dvol) & \\
        & \hat{c}(\dvol)
    \end{bmatrix}\begin{bmatrix}
        d+d^* & \omega\\
        \omega^* & -d-d^*
    \end{bmatrix} \\
    =\ & \ind_2\begin{bmatrix}
        \hat{c}(\dvol)\omega^* & -\hat{c}(\dvol)(d+d^*)\\
        \hat{c}(\dvol)(d+d^*) & \hat{c}(\dvol)\omega
    \end{bmatrix} \\
    =\ & \ind_2 \left(\hspace{-0.5mm}\begin{bmatrix}
        \hat{c}(\dvol)\dfrac{1}{2}(\omega^*-\omega) & -\hat{c}(\dvol)(d+d^*)\\
        \hat{c}(\dvol)(d+d^*) & \hat{c}(\dvol)\dfrac{1}{2}(\omega-\omega^*)
    \end{bmatrix} + \dfrac{1}{2}\begin{bmatrix}
        \hat{c}(\dvol)(\omega^*+\omega) & \\
        & \hat{c}(\dvol)(\omega+\omega^*)
    \end{bmatrix}\hspace{-0.5mm}\right)  \\
    =\ & \ind_2\begin{bmatrix}
        \hat{c}(\dvol)\dfrac{1}{2}(\omega^*-\omega) & -\hat{c}(\dvol)(d+d^*)\\
        \hat{c}(\dvol)(d+d^*) & \hat{c}(\dvol)\dfrac{1}{2}(\omega-\omega^*)
    \end{bmatrix} \\
    =\ & \ind_2\begin{bmatrix}
        \dfrac{1}{2}(\omega^*-\omega) & -d-d^*\\
        d+d^* & \dfrac{1}{2}(\omega-\omega^*)
    \end{bmatrix}. 
    \end{align*}
    The second-to-last equality is justified by Proposition \ref{continuous family}. 
\end{proof}

\section{Symplectic Witten deformation}\label{Symplectic Witten deformation section}
In this section, we study the symplectic Witten deformation of the skew-adjoint operator (\ref{before witten def}) on $C^\even(M,\omega) \coloneqq \Omega^{\even}(M)\oplus\Omega^{\odd}(M)$. 

We let $V$ be a nondegenerate smooth vector field on $M$. This means around each zero point $p$ of $V$, given any small local chart with coordinates $x_1,y_1,\cdots,x_{2n},y_{2n}$ satisfying $x_1(p) = \cdots = y_{2n}(p) = 0,$ there is an $\mathbb{R}^{4n}$-valued smooth function $B$ on the chart with order $$O(x_1^2+y_1^2+\cdots+x_{2n}^2+y_{2n}^2)$$ and a matrix $A\in GL_{4n}(\mathbb{R})$ such that 
\begin{align}\label{matrix A}
    V(x_1,y_1,\cdots,x_{2n},y_{2n}) = \begin{bmatrix}
\partial_{x_1},\partial_{y_1},\cdots,\partial_{x_{2n}},\partial_{y_{2n}}
\end{bmatrix}\left(A\begin{bmatrix}
    x_1\\
    y_1\\
    \vdots\\
    x_{2n}\\
    y_{2n}
\end{bmatrix}+B\right).
\end{align}
 Here, $\partial_{x_i}$ and $\partial_{y_i}$ are the local coordinate vector fields. For convenience, we let $${\mathbf{x}} = \begin{bmatrix}
    x_1\\
    y_1\\
    \vdots\\
    x_{2n}\\
    y_{2n}
\end{bmatrix},\ {\mathbf{x}}^{\mathrm{t}} = [x_1,y_1,\cdots,x_{2n},y_{2n}], $$ and $|{\mathbf{x}}| = \left(x_1^2+y_1^2+\cdots+x_{2n}^2+y_{2n}^2\right)^{1/2}$. In particular, for any column $\mathbf{v} = \begin{bmatrix}
        \varphi_1,\cdots,\varphi_{4n}
    \end{bmatrix}^\mathrm{t}$, we let $|\mathbf{v}| = \sqrt{\varphi_1^2+\cdots+\varphi_{4n}^2}.$ 
\begin{lemma}\label{replacing vector field lemma}
    There is a smooth vector field $X$ on $M$ such that the zero set of $X$ is the same as the zero set of $V$, and 
    $$X = \begin{bmatrix}
\partial_{x_1},\partial_{y_1},\cdots,\partial_{x_{2n}},\partial_{y_{2n}}
\end{bmatrix}A{\mathbf{x}}$$
    near each zero point $p$.  
\end{lemma}
\begin{proof}
We will use a cutoff function to modify the vector field $V$ near its zeros so that it agrees with a standard local model. We find a constant $C>0$ such that 
    \begin{align}\label{vector field replacing 1}
        |B|\leqslant C\left|\mathbf{x}\right|^2
    \end{align}
    on the local chart centered at $p$. Viewing $A$ as an operator on the linear space $\mathbb{R}^{4n}$, we let $\|A\|$ be its operator norm. Then, we choose a bump function $\sigma$ such that 
    \begin{align}\label{vector field replacing 2}
        \supp(\sigma)\subseteq\left\{(x_1,\cdots,y_{2n}): \left(x_1^2+\cdots+y_{2n}^2\right)^{1/2}<\|A^{-1}\|^{-1}\cdot C^{-1}\right\}
    \end{align}
    and $\sigma = 1$ near $p$. Now, we show that 
    \begin{align*}
        X =\ & \sigma V+(1-\sigma)\begin{bmatrix}
\partial_{x_1},\partial_{y_1},\cdots,\partial_{x_{2n}},\partial_{y_{2n}}
\end{bmatrix}A{\mathbf{x}}\\
=\ & \begin{bmatrix}
\partial_{x_1},\partial_{y_1},\cdots,\partial_{x_{2n}},\partial_{y_{2n}}
\end{bmatrix}\left(A{\mathbf{x}}+(1-\sigma)B\right)
    \end{align*} 
    is the vector field we need. 
    Actually, by (\ref{vector field replacing 1}) and (\ref{vector field replacing 2}), 
    \begin{align*}
       \left|{\mathbf{x}}+(1-\sigma)A^{-1}B\right|
        \geqslant |\mathbf{x}|-\|A^{-1}\|\cdot C\cdot |\mathbf{x}|^2
        \geqslant 0. 
    \end{align*}
    The last ``$\geqslant$'' becomes ``$=$'' if and only if $\mathbf{x}$ is zero. 
    Thus, $$A\mathbf{x}+(1-\sigma)B = A\cdot\left({\mathbf{x}}+(1-\sigma)A^{-1}B\right) = 0$$ if and only if at the zero point $p$ of $V$. Therefore, the zero set of $X$ coincides that of $V$. 
\end{proof}

Inspired by \cite[Section 2]{witten}, \cite[Chapters VIII-X]{bismutandlebeau}, \cite[Section 7.3]{wittendeformationweipingzhang}, and \cite[Section 2]{zhangcountingmod2indexkervairesemi}, for a parameter $T>0$, we use the vector field $X$ to set up the Witten deformation 
\begin{align}\label{the witten deformation of the skew adjoint operator constructed above and use later}
\mathbb{D}_{T}\coloneqq \begin{bmatrix}
    \dfrac{1}{2}(\omega^*-\omega) & -d-d^*-T\hat{c}(X)\\ 
    d+d^*+T\hat{c}(X) & \dfrac{1}{2}(\omega-\omega^*)
\end{bmatrix}
\end{align}
of the operator (\ref{before witten def}) on $\Omega^\even(M)\oplus\Omega^\odd(M)$. Let $\varepsilon>0$ be a sufficiently small number. Around each zero point $p$ of $X$, we choose the chart 
\begin{align}\label{chart 4 epsilon}
    U = \{(x_1,\cdots,y_{2n}): x_1^2+y_1^2+\cdots+x_{2n}^2+y_{2n}^2 < (4\varepsilon)^2\}
\end{align}
centered at $p$ such that the following three items hold true at the same time: 
\begin{enumerate}[label = (\arabic*)]
    \item $\omega\vert_U = dx_1\wedge dy_1+\cdots+dx_{2n}\wedge dy_{2n}$;
    \item The metric $g(\cdot,\cdot)\vert_U = dx_1^2+dy_1^2+\cdots+dx_{2n}^2+dy_{2n}^2$;
    \item $X\vert_U = \begin{bmatrix}
\partial_{x_1},\partial_{y_1},\cdots,\partial_{x_{2n}},\partial_{y_{2n}}
\end{bmatrix}A{\mathbf{x}}$.
\end{enumerate}
Actually, we can obtain the above (1)-(3) in this way: By \cite[Theorem 8.1]{daSilva2008}, we first choose a Darboux chart $U$ centered at the zero point $p$ with coordinates $x_1,y_1,\cdots,x_{2n},y_{2n}$ such that 
$$\omega\vert_U = dx_1\wedge dy_1+\cdots+dx_{2n}\wedge dy_{2n}.$$ Second, we construct a metric $g'$ on $M$ such that $$g'(\cdot,\cdot)\vert_U = dx_1^2+dy_1^2+\cdots+dx_{2n}^2+dy_{2n}^2.$$ Third, 
according to the proof of \cite[Proposition 12.3]{daSilva2008}, we use the polar decomposition together with $g'$ to construct the almost complex structure $J$. Then, we let $g(\cdot,\cdot) = \omega(\cdot,J\cdot)$. The two metrics $g(\cdot,\cdot)$ and $g'(\cdot,\cdot)$ are different, but checking the polar decomposition, we have $$g(\cdot,\cdot)\vert_U = g'(\cdot,\cdot)\vert_U = dx_1^2+dy_1^2+\cdots+dx_{2n}^2+dy_{2n}^2.$$
Finally, the vector field $X$ is guaranteed by Lemma \ref{replacing vector field lemma}. 

Let ${\mathbf{e}_i} = \begin{bmatrix}
    0, 
    \cdots, 
    1, 
    0, 
    \cdots, 
    0
\end{bmatrix}^\mathrm{t}$ (the $(2i-1)$-th entry is $1$) and ${\mathbf{f}_i} = \begin{bmatrix}
    0, 
    \cdots, 
    0, 
    1, 
    \cdots, 
    0
\end{bmatrix}^\mathrm{t}$ (the $2i$-th entry is $1$). Then, inside $U$, we find that  
\begin{align*}
     & (d+d^*+T\hat{c}(X))^2 \\
    =\  & -\sum_{i = 1}^{2n}\partial_{x_i}^2-\sum_{i = 1}^{2n}\partial_{y_i}^2 + T\sum_{i = 1}^{2n}c(\partial_{x_i})\hat{c}\left(\begin{bmatrix}
        \partial_{x_1},\partial_{y_1},\cdots,\partial_{x_{2n}},\partial_{y_{2n}}
    \end{bmatrix}A{\mathbf{e}_i}\right)\\
    & + T\sum_{i = 1}^{2n}c(\partial_{y_i})\hat{c}\left(\begin{bmatrix}
        \partial_{x_1},\partial_{y_1},\cdots,\partial_{x_{2n}},\partial_{y_{2n}}
    \end{bmatrix}A{\mathbf{f}_i}\right) + T^2\mathbf{x}^\mathrm{t}
    A^*A{\mathbf{x}}.
\end{align*}

Now, on $\mathbb{R}^{4n}$ (coordinates denoted by $x_1,y_1,\cdots,x_{2n},y_{2n}$), we let $$X_0 = \begin{bmatrix}
        \partial_{x_1},\partial_{y_1},\cdots,\partial_{x_{2n}}, \partial_{y_{2n}}
    \end{bmatrix}A{\mathbf{x}}.$$ Meanwhile, using the standard Euclidean metric $$g_0 \coloneqq dx_1^2+dy_1^2+\cdots+dx_{2n}^2+dy_{2n}^2$$ on $\mathbb{R}^{4n}$, we have the $L^2$-norm (also, the inner product)
\begin{align}\label{norm on rn no pair}
    \left\|\alpha \right\| = \left(\int_{\mathbb{R}^{4n}}g_0(\alpha,\alpha)dx_1\wedge dy_1\wedge\cdots\wedge dx_{2n}\wedge dy_{2n}\right)^{1/2}
\end{align}
on the space $\Omega^k(\mathbb{R}^{4n})$ of smooth $k$-forms on $\mathbb{R}^{4n}$. Also, for the standard symplectic form 
$$\omega_0 = dx_1\wedge dy_1+\cdots+dx_{2n}\wedge dy_{2n}$$
on $\mathbb{R}^{4n}$, we let 
$$\omega_0^*\lrcorner = \partial_{y_1}\lrcorner\partial_{x_1}\lrcorner+\cdots+\partial_{y_{2n}}\lrcorner\partial_{x_{2n}}\lrcorner$$ be the adjoint of $\omega_0\wedge$\ . 

Let $L$ be the operator with the expression 
\begin{align*}
\begin{split}
    & -\sum_{i = 1}^{2n}\partial_{x_i}^2-\sum_{i = 1}^{2n}\partial_{y_i}^2 + T\sum_{i = 1}^{2n}c(\partial_{x_i})\hat{c}\left(\begin{bmatrix}
        \partial_{x_1},\partial_{y_1},\cdots,\partial_{x_{2n}},\partial_{y_{2n}}
    \end{bmatrix}A{\mathbf{e}_i}\right)\\
    & + T\sum_{i = 1}^{2n}c(\partial_{y_i})\hat{c}\left(\begin{bmatrix}
        \partial_{x_1},\partial_{y_1},\cdots,\partial_{x_{2n}},\partial_{y_{2n}}
    \end{bmatrix}A{\mathbf{f}_i}\right) + T^2\mathbf{x}^\mathrm{t}A^*A{\mathbf{x}}
\end{split}
\end{align*}
but defined on the space $\bigoplus_{k = 0}^{4n}\Omega^k(\mathbb{R}^{4n})$ of smooth forms on $\mathbb{R}^{4n}$. Like in \cite[(4.23)]{wittendeformationweipingzhang}, we let 
\begin{align*}
    L' =\  & -\sum_{i = 1}^{2n}\partial_{x_i}^2-\sum_{i = 1}^{2n}\partial_{y_i}^2 - T\cdot\text{trace}\left(\sqrt{A^*A}\right) + T^2\mathbf{x}^\mathrm{t}A^*A{\mathbf{x}}
\end{align*}
and 
\begin{align*}
    L'' =\ \text{trace}\left(\sqrt{A^*A}\right) & + \sum_{i = 1}^{2n}c(\partial_{x_i})\hat{c}\left(\begin{bmatrix}
        \partial_{x_1},\partial_{y_1},\cdots,\partial_{x_{2n}},\partial_{y_{2n}}
    \end{bmatrix}A{\mathbf{e}_i}\right) \\
    & + \sum_{i = 1}^{2n}c(\partial_{y_i})\hat{c}\left(\begin{bmatrix}
        \partial_{x_1},\partial_{y_1},\cdots,\partial_{x_{2n}},\partial_{y_{2n}}
    \end{bmatrix}A{\mathbf{f}_i}\right).
\end{align*}
Then, $L = L'+T\cdot L''$. Actually, $L'$ is the (rescaled) harmonic oscillator \cite[Chapter 8, Section 6]{taylor2010partialvol2} on the space of square-integrable functions on $\mathbb{R}^{4n}$, and $L''$ is a nonnegative operator on the space 
\begin{align*}
        &\text{span}_\mathbb{R}\{dx_{i_1}\wedge\cdots\wedge dx_{i_r}\wedge dy_{j_1}\wedge\cdots\wedge dy_{j_s}:\ \ \ \ \ \ \ \ \ \ \ \ \ \ \ \ \ \ \ \ \ \ \ \ \ \ \ \ \ \ \ \ \ \ \ \ \ \ \ \ \ \ \ \ \ \ \ \  \\ & \ \ \ \ \ \ \ \ \ \ \ 0\leqslant i_1<\cdots< i_r\leqslant 2n, 0\leqslant j_1<\cdots<j_s\leqslant 2n, 0\leqslant r,s\leqslant 2n\}.
    \end{align*}
Based on the properties \cite[Chapter 8, Section 6]{taylor2010partialvol2} of harmonic oscillators, we have the following proposition \cite[Proposition 4.9]{wittendeformationweipingzhang}. The proof is given in \cite[(4.23)-(4.25) and Lemma 4.8]{wittendeformationweipingzhang}.  
\begin{proposition}\label{proposition of harmonic oscillator with forms}
    For any $T>0$, the kernel of $L$ is $1$-dimensional and generated by
    \begin{align}\label{kernel element one generator}
        \rho = \exp\left(-\dfrac{T}{2}{\mathbf{x}}^\mathrm{t}\sqrt{A^*A}{\mathbf{x}}\right)\cdot\delta\ ,
    \end{align}
    where $\delta$ is a certain linear combination (with real coefficients independent of $T$) of 
    \begin{align*}
        \{dx_{i_1}\wedge\cdots\wedge dx_{i_r}\wedge dy_{j_1}\wedge\cdots\wedge dy_{j_s}: \ \ \ \ \ \ \ \ \ \ \ \ \ \ \ \ \ \ \ \ \ \ \ \ \ \ \ \ \ \ \ \ \ \ \ \ \ \ \ \ \ \ \  \\ 0\leqslant i_1<\cdots< i_r\leqslant 2n, 0\leqslant j_1<\cdots<j_s\leqslant 2n, 0\leqslant r,s\leqslant 2n\}.
    \end{align*}
    Also, the grading of $\delta$ is even (resp. odd) if $\det A > 0$ (resp. $\det A < 0$). Moreover, each nonzero eigenvalue of $L$ has the expression $\alpha\cdot T$ ($\alpha$ is a positive constant independent of $T$). 
\end{proposition}

Notice that $\omega_0^*\lrcorner - \omega_0\wedge$ is skew-symmetric, we have: 
\begin{proposition}\label{functional calculus inverse}
    There exists a unique smooth form $\eta$ on $\mathbb{R}^{4n}$ such that 
    $$\dfrac{1}{2}(\omega_0^*-\omega_0)\rho = \left(d+d^*+T\hat{c}\left(X_0\right)\right)\eta$$
    and $\eta\perp\rho$. 
    Here, $d+d^*$ is defined on $\bigoplus_{k = 0}^{4n}\Omega^k(\mathbb{R}^{4n})$ according to the $L^2$-norm of forms.
\end{proposition}
\begin{proof}
    Recall the definition (\ref{kernel element one generator}) of $\rho$. We notice that $$g_0\left((\omega_0^*-\omega_0)\rho,\rho\right) = 0.$$
    Therefore, $\dfrac{1}{2}(\omega_0^*-\omega_0)\rho$ is orthogonal to 
    the kernel of $L$ on $\bigoplus_{k = 0}^{4n}\Omega^k(\mathbb{R}^{4n})$. Then, since $d+d^*+T\hat{c}\left(X_0\right)$ preserves the eigenspaces of $L$, we find
    \begin{align}\label{applying the inverse}
        \eta = L^{-1}\circ\left(d+d^*+T\hat{c}\left(X_0\right)\right)\left(\dfrac{1}{2}(\omega_0^*-\omega_0)\rho\right).
    \end{align}
    Here, this $L^{-1}$ is the inverse of $L$ restricted to the orthogonal complement of $\ker(L)$. The kernel of $L$ is given by Proposition \ref{proposition of harmonic oscillator with forms}. See \cite[(10.17)]{bismutandlebeau} for more details about the inverse map $L^{-1}$. 
\end{proof}

The next proposition will be used in the estimates of the spectrum of $-\mathbb{D}_T^2$. 
\begin{proposition}
    There is a constant $C_1\geqslant 0$ independent of $T$ such that
    \begin{align}\label{key estimate rho eta}
       \|\eta\| = C_1 T^{-1/2}\cdot\|\rho\|.
    \end{align}
    Here, the $L^2$-norm is that on the space of forms on $\mathbb{R}^{4n}$. 
\end{proposition}
\begin{proof}
If $\eta = 0$, we choose $C_1 = 0$. Now, if $\eta\neq 0$, by Proposition \ref{functional calculus inverse}, $\dfrac{1}{2}(\omega_0^*-\omega_0)\rho\neq 0$. 
    Then, we look at (\ref{applying the inverse}). We write $\dfrac{1}{2}(\omega_0^*-\omega_0)\rho$ into a finite sum of eigenvectors of $L$: 
    $$\dfrac{1}{2}(\omega_0^*-\omega_0)\rho = \sum_{i}K_i\cdot\exp\left(-\dfrac{T}{2}\mathbf{x}^\mathrm{t}\sqrt{A^*A}{\mathbf{x}}\right)\cdot\delta_i,$$
    where each $K_i$ is a constant independent of $T$, and each $\delta_i$ is an eigenvector of $L''$ on 
    \begin{align*}
        &\text{span}_\mathbb{R}\{dx_{i_1}\wedge\cdots\wedge dx_{i_r}\wedge dy_{j_1}\wedge\cdots\wedge dy_{j_s}: \ \ \ \ \ \ \ \ \ \ \ \ \ \ \ \ \ \ \ \ \ \ \ \ \ \ \ \ \ \ \ \ \ \ \ \ \ \ \ \ \ \ \ \ \ \ \ \ \ \\ & \ \ \ \ \ \ \ \ \ \ \ 0\leqslant i_1<\cdots< i_r\leqslant 2n, 0\leqslant j_1<\cdots<j_s\leqslant 2n, 0\leqslant r,s\leqslant 2n\}
    \end{align*}
    associated with an eigenvalue $\lambda_i>0$. These $\delta_i$'s and $\lambda_i$'s satisfy $$g_0(\delta_i,\delta_j) = 0\text{\ \ and \ $\lambda_i\neq\lambda_j$\  when $i\neq j$}. $$
    Then, we apply $$L^{-1}\circ(d+d^*+T\hat{c}\left(X_0\right))$$ to $\dfrac{1}{2}(\omega_0^*-\omega_0)\rho$. Since $d+d^*+T\hat{c}\left(X_0\right)$ preserves the eigenspaces of $L$, 
    we obtain 
    \begin{align*}
        & L^{-1}\circ(d+d^*+T\hat{c}\left(X_0\right))\left(\dfrac{1}{2}(\omega_0^*-\omega_0)\rho\right) \\
        =\ & \sum_{i}\dfrac{1}{\lambda_iT}\cdot(d+d^*+T\hat{c}\left(X_0\right))\left(K_i\cdot\exp\left(-\dfrac{T}{2}\mathbf{x}^\mathrm{t}\sqrt{A^*A}{\mathbf{x}}\right)\cdot\delta_i\right).
    \end{align*}
One step further, considering the effect of $d+d^*+T\hat{c}(X_0)$, we find 
    \begin{align*}
       \eta =\ & L^{-1}\circ(d+d^*+T\hat{c}\left(X_0\right))\left(\dfrac{1}{2}(\omega_0^*-\omega_0)\rho\right)\\
        =\ & \sum_{i}\dfrac{1}{\lambda_i T}\cdot T\cdot\exp\left(-\dfrac{T}{2}\mathbf{x}^\mathrm{t}\sqrt{A^*A}{\mathbf{x}}\right)\cdot\left(\sum_{j = 1}^{2n} K_{ij}\cdot x_j\cdot\delta_{ij} + \sum_{j = 1}^{2n} \widetilde{K}_{ij}\cdot y_j\cdot\widetilde{\delta}_{ij}\right)\\
        =\ & \sum_{i}\dfrac{1}{\lambda_i}\cdot\exp\left(-\dfrac{T}{2}\mathbf{x}^\mathrm{t}\sqrt{A^*A}{\mathbf{x}}\right)\cdot\left(\sum_{j = 1}^{2n} K_{ij}\cdot x_j\cdot\delta_{ij} + \sum_{j = 1}^{2n} \widetilde{K}_{ij}\cdot y_j\cdot\widetilde{\delta}_{ij}\right),
    \end{align*}
    where $K_{ij}$'s and $\widetilde{K}_{ij}$'s are constants independent of $T$, and $\delta_{ij}$'s and $\widetilde{\delta}_{ij}$'s are certain linear combinations (with real coefficients independent of $T$) of 
    \begin{align*}
        & \{dx_{i_1}\wedge\cdots\wedge dx_{i_r}\wedge dy_{j_1}\wedge\cdots\wedge dy_{j_s}:\ \ \ \ \ \ \ \ \ \ \ \ \ \ \ \ \ \ \ \ \ \ \ \ \ \ \ \ \ \ \ \ \ \ \ \ \ \ \ \ \ \ \ \ \ \ \ \ \ \ \ \ \ \ \ \\ & \ \ \ \ \ \ \ \ \ \  0\leqslant i_1<\cdots< i_r\leqslant 2n, 0\leqslant j_1<\cdots<j_s\leqslant 2n, 0\leqslant r,s\leqslant 2n\}.
    \end{align*}
   Thus, (\ref{key estimate rho eta}) is essentially the relation between $$\left(\int_{\mathbb{R}^{4n}}x_i^2\exp\left(-{T}|\mathbf{x}|^2\right)dx_1dy_1\cdots dx_{2n}dy_{2n}\right)^{1/2} = \dfrac{\pi^{n}}{T^{n}}\cdot\dfrac{1}{\sqrt{2T}}$$
   and 
   $$\left(\int_{\mathbb{R}^{4n}}\exp\left(-{T}|\mathbf{x}|^2\right)dx_1dy_1\cdots dx_{2n}dy_{2n}\right)^{1/2} = \dfrac{\pi^{n}}{T^{n}}.$$
  Their ratio gives us the factor $T^{-1/2}$. 
\end{proof}
\begin{remark}
\normalfont
    In the standard Witten deformation, the form $\rho$ plays the role as a model for eigenforms associated with small eigenvalues of the deformed Laplacian. In this paper, the pair $(\rho, \eta)$ plays a similar role in the mapping cone Witten deformation. 
\end{remark}

Now, using (\ref{norm on rn no pair}), we define 
the $L^2$-norm (also, the inner product)
\begin{align}\label{inner product on rn}
    \left\|\begin{bmatrix}
        \alpha\\
        \beta
    \end{bmatrix}\right\| =  \left(\|\alpha\|^2+\|\beta\|^2\right)^{1/2}
\end{align}
on $\Omega^\even(\mathbb{R}^{4n})\oplus\Omega^\odd(\mathbb{R}^{4n})$.
Recall the matrix $A$ in (\ref{matrix A}) associated with the zero point $p$. When $\det A > 0$, we study the orthogonal complement of 
$$\text{span}_{\mathbb{R}}\left\{\begin{bmatrix}
    \rho\\
    \eta
\end{bmatrix}\right\}$$
in $\Omega^\even(\mathbb{R}^{4n})\oplus\Omega^\odd(\mathbb{R}^{4n})$ under the inner product induced by (\ref{inner product on rn}).
Let $\begin{bmatrix}
    \alpha\\
    \beta
\end{bmatrix}\in\Omega^\even(\mathbb{R}^{4n})\oplus\Omega^\odd(\mathbb{R}^{4n})$ be an $L^2$-element such that $\begin{bmatrix}
        \alpha\\
        \beta
    \end{bmatrix}$ and $\begin{bmatrix}
        \rho\\
        \eta
    \end{bmatrix}$ are orthogonal to each other. We write
 $$\alpha = r\rho+\alpha'\text{\ \ and\ \ } \beta = s\eta+\beta'$$
such that $\alpha'\perp\rho$ and $\beta'\perp\eta$. Then, we have 
 \begin{align}\label{dependence of r and s}
     r\|\rho\|^2+s\|\eta\|^2 = 0.
 \end{align}
Let $\|\cdot\|_1$ be the 1st Sobolev norm (See \cite[Definition 10.2.7]{nicolaescu2020lectures}) induced by (\ref{norm on rn no pair}). If $\|\alpha\|_1<\infty$ and $\|\beta\|_1<\infty$, we find that $\|\alpha'\|_1<\infty$,  $\|\beta'\|_1<\infty$, and then
 \begin{align}\label{two situations of eta}
     &\left\|\begin{bmatrix}
    \dfrac{1}{2}(\omega_0^*-\omega_0) & -d-d^*-T\hat{c}\left(X_0\right)\\  
    d+d^*+T\hat{c}\left(X_0\right) & \dfrac{1}{2}(\omega_0-\omega_0^*)
\end{bmatrix}\begin{bmatrix}
    \alpha \nonumber \\
    \beta
\end{bmatrix}\right\|\\ \nonumber
\geqslant\ & \left\|\begin{bmatrix}
    & -d-d^*-T\hat{c}(X_0)\\
    d+d^*+T\hat{c}(X_0) & 
\end{bmatrix}\begin{bmatrix}
    r\rho +\alpha'\\
    s\eta+\beta'
\end{bmatrix}\right\|  \\\nonumber
& \ -\left\|\dfrac{1}{2}\begin{bmatrix}
    (\omega_0^*-\omega_0) & \\
    & (\omega_0-\omega_0^*)
\end{bmatrix}\begin{bmatrix}
    r\rho + \alpha'\\
    s\eta + \beta'
\end{bmatrix}\right\|\\ \nonumber
\geqslant\ & \left\|\begin{bmatrix}
    (-d-d^*-T\hat{c}(X_0))(s\eta+\beta')\\
    (d+d^*+T\hat{c}(X_0))(r\rho+\alpha')
\end{bmatrix}\right\|-C_2\left\|\begin{bmatrix}
    r\rho+\alpha'\\
    s\eta+\beta'
\end{bmatrix}\right\|\\ \nonumber
=\ & \left(\|(d+d^*+T\hat{c}(X_0))(s\eta+\beta')\|^2 + \|(d+d^*+T\hat{c}(X_0))\alpha'\|^2\right)^{1/2} - C_2\left\|\begin{bmatrix}
    r\rho+\alpha'\\
    s\eta+\beta'
\end{bmatrix}\right\|  \\\nonumber
&\ \text{\ \ \ (Since $(d+d^*+T\hat{c}(X_0))\rho = 0$)}\\ \nonumber
\geqslant\ & C_3\sqrt{T}\|s\eta+\beta'\|+C_3\sqrt{T}\|\alpha'\|-C_2\left\|\begin{bmatrix}
    r\rho+\alpha'\\
    s\eta+\beta'
\end{bmatrix}\right\| \\ \nonumber
& (\text{By $\|\alpha'\|_1<\infty$, $\|\beta'\|_1<\infty$, and Proposition \ref{proposition of harmonic oscillator with forms}}) \\[8pt] \nonumber
=\ & C_3\sqrt{T}\sqrt{\|s\eta\|^2+\|\beta'\|^2} + C_3\sqrt{T}\|\alpha'\|-C_2\left\|\begin{bmatrix}
    r\rho+\alpha'\\
    s\eta+\beta'
\end{bmatrix}\right\|\\ 
\geqslant\ & C_4\sqrt{T}\|s\eta\|+C_4\sqrt{T}\|\beta'\|+C_3\sqrt{T}\|\alpha'\|-C_2\left\|\begin{bmatrix}
    r\rho+\alpha'\\
    s\eta+\beta'
\end{bmatrix}\right\|.
\end{align}
We approach (\ref{two situations of eta}) in two different cases: 

\vspace{+3mm}
\noindent Case 1: When $\eta \neq 0$:
\begin{align*}
& \text{The last line of (\ref{two situations of eta})}\\
=\ & \dfrac{1}{2}C_4\sqrt{T}\|s\eta\|+\dfrac{1}{2}C_4\sqrt{T}\dfrac{|r|\cdot\|\rho\|^2}{\|\eta\|^2}\|\eta\| + C_4\sqrt{T}\|\beta'\|  + C_3\sqrt{T}\|\alpha'\| - C_2\left\|\begin{bmatrix}
    r\rho+\alpha'\\
    s\eta+\beta'
\end{bmatrix}\right\|\\
=\ & \dfrac{1}{2}C_4\sqrt{T}\|s\eta\|+\dfrac{1}{2}C_4\sqrt{T}\|r\rho\|C_1^{-1}\sqrt{T} + C_4\sqrt{T}\|\beta'\| + C_3\sqrt{T}\|\alpha'\| - C_2\left\|\begin{bmatrix}
    r\rho+\alpha'\\
    s\eta+\beta'
\end{bmatrix}\right\|\\
\geqslant\ & C_5\sqrt{T}\left\|\begin{bmatrix}
    r\rho+\alpha'\\
    s\eta+\beta'
\end{bmatrix}\right\|.
\end{align*}
Case 2: When $\eta = 0$: By (\ref{dependence of r and s}), we find $r = 0$ and then
\begin{align*}
\text{The last line of (\ref{two situations of eta})}
=\ & C_4\sqrt{T}\|\beta'\|+C_3\sqrt{T}\|\alpha'\|-C_2\left\|\begin{bmatrix}
    \alpha'\\
    \beta'
\end{bmatrix}\right\|\\ 
\geqslant\ & C_5\sqrt{T}\left\|\begin{bmatrix}
    \alpha'\\
    \beta' 
\end{bmatrix}\right\|\\
=\ & C_5\sqrt{T}\left\|\begin{bmatrix}
    r\rho+\alpha'\\
    s\eta+\beta' 
\end{bmatrix}\right\|.
 \end{align*}
When $\det A < 0$,  we replace $\begin{bmatrix}
     \rho\\
     \eta
 \end{bmatrix}$ by $\begin{bmatrix}
     \eta\\
     \rho
 \end{bmatrix}$ and repeat the above steps. Then, we summarize: 
 \begin{proposition}\label{higher part estimate}
     There exists a constant $C_5>0$ such that when $\det A > 0$ (resp. when $\det A < 0$), for all sufficiently large $T$, we have
     \begin{align*}
     \left\|\begin{bmatrix}
    \dfrac{1}{2}(\omega_0^*-\omega_0) & -d-d^*-T\hat{c}\left(X_0\right)\\ 
    d+d^*+T\hat{c}\left(X_0\right) & \dfrac{1}{2}(\omega_0-\omega_0^*)
\end{bmatrix}\begin{bmatrix}
    \alpha\\
    \beta
\end{bmatrix}\right\|\geqslant C_5\sqrt{T}\left\|\begin{bmatrix}
          \alpha\\
          \beta
      \end{bmatrix}\right\|
     \end{align*}
     whenever $\begin{bmatrix}
         \alpha\\
         \beta
     \end{bmatrix}\in\Omega^\even(\mathbb{R}^{4n})\oplus\Omega^\odd(\mathbb{R}^{4n})$ is orthogonal to $\begin{bmatrix}
         \rho\\
         \eta
     \end{bmatrix}$ (resp. $\begin{bmatrix}
         \eta\\
         \rho
     \end{bmatrix}$) and satisfies $\|\alpha\|_1<\infty$ and $\|\beta\|_1<\infty$. 
 \end{proposition}
Following \cite{bismutandlebeau} and \cite{wittendeformationweipingzhang}, based on Propositions \ref{proposition of harmonic oscillator with forms}-\ref{higher part estimate}, we apply the asymptotic analysis to carry out the estimates about $\mathbb{D}_T$. Recall (\ref{chart 4 epsilon}) the chart $U$ around each zero point $p$ of $X$. For each zero point $p$, we pick a bump function $\gamma: M\to\mathbb{R}$ such that $$\supp(\gamma)\subseteq U(2\varepsilon) \coloneqq \{(x_1,\cdots,y_{2n}): x_1^2+\cdots+y_{2n}^2 < (2\varepsilon)^2\},$$ and $\gamma = 1$ on $$U(\varepsilon) \coloneqq \{(x_1,\cdots,y_{2n}): x_1^2+\cdots+y_{2n}^2 < \varepsilon^2\}.$$ Furthermore, for each zero point $p$, we let
$$\rho_p = \gamma\cdot\exp\left(-\dfrac{T}{2}\mathbf{x}^\mathrm{t}\sqrt{A^*A}{\mathbf{x}}\right)\cdot\delta$$
and $\eta_p = \gamma\cdot\eta$. As \cite[Definition 9.4]{bismutandlebeau} and \cite[(4.36)]{wittendeformationweipingzhang}, we let $E_T$ be the linear space 
$$\text{span}_\mathbb{R}\left\{\begin{bmatrix}
    \rho_p\\
    \eta_p
\end{bmatrix} (\text{resp. $\begin{bmatrix}
    \eta_p\\
    \rho_p
\end{bmatrix}$}): p \text{\ is a zero point of\ }X \text{\ with $\det A > 0$} \text{\ (resp. $\det A < 0$)}\right\},$$
and $E_T^\perp$ be the orthogonal complement of $E_T$ in $\Omega^\even(M)\oplus\Omega^\odd(M)$. Then, we let $p_T$ (resp. $p_T^\perp$) be the orthogonal projection from $\Omega^\even(M)\oplus\Omega^\odd(M)$ to $E_T$ (resp. $E_T^\perp$). 

Recall the operator $$\mathbb{D}_{T}\coloneqq \begin{bmatrix}
    \dfrac{1}{2}(\omega^*-\omega) & -d-d^*-T\hat{c}(X)\\ 
    d+d^*+T\hat{c}(X) & \dfrac{1}{2}(\omega-\omega^*)
\end{bmatrix}$$ on $\Omega^\even(M)\oplus\Omega^\odd(M)$. 
Then, there is a constant $C_6>0$ such that
\begin{align*}
    \left\|\mathbb{D}_{T}\begin{bmatrix}
        \rho_p\\
        \eta_p
    \end{bmatrix}\right\|
    =\ & \left\|\begin{bmatrix}
    \dfrac{1}{2}(\omega^*-\omega) & -d-d^*-T\hat{c}(X)\\ 
    d+d^*+T\hat{c}(X) & \dfrac{1}{2}(\omega-\omega^*)
\end{bmatrix}\begin{bmatrix}
    \gamma\rho\\
    \gamma\eta
\end{bmatrix}\right\|\\
=\ & \left\|\begin{bmatrix}
    -c(d\gamma)\eta\\
    c(d\gamma)\rho+\dfrac{1}{2}(\omega-\omega^*)\gamma\eta
\end{bmatrix}\right\|\\
    \leqslant\ &\left\|\begin{bmatrix}
        c(d\gamma)\eta\\
        c(d\gamma)\rho
    \end{bmatrix}\right\| + \left\|\begin{bmatrix}
        0\\
        \dfrac{1}{2}(\omega-\omega^*)\gamma\eta
    \end{bmatrix}\right\|\\
    \leqslant\ & C_6\left\|\begin{bmatrix}
        \rho_p\\
        \eta_p
    \end{bmatrix}\right\|
\end{align*}
 when $T$ is sufficiently large. Summarizing this estimate, we get: 
\begin{proposition}\label{spectral left}
    There is a constant $C_6>0$ such that when $T$ is sufficiently large, 
    $$\left\|\mathbb{D}_{T}\begin{bmatrix}
        \alpha\\
        \beta
    \end{bmatrix}\right\|\leqslant C_6\cdot\left\|\begin{bmatrix}
        \alpha\\
        \beta
    \end{bmatrix}\right\|$$
    for all $\begin{bmatrix}
        \alpha\\
        \beta
    \end{bmatrix}\in E_T$. 
\end{proposition}
Now, if $\begin{bmatrix}
    \alpha\\
    \beta
\end{bmatrix}\in E_T^\perp$, we have the following estimate similar to those in \cite[Theorem 9.11]{bismutandlebeau} and \cite[Proposition 4.12]{wittendeformationweipingzhang}: 
\begin{proposition}\label{spectral right}
    There exists a constant $C_{11}>0$ such that when $T$ is sufficiently large, 
    $$\left\|\mathbb{D}_{T}\begin{bmatrix}
        \alpha\\
        \beta
    \end{bmatrix}\right\|\geqslant C_{11}\sqrt{T}\left\|\begin{bmatrix}
        \alpha\\
        \beta
    \end{bmatrix}\right\|$$ for all $\begin{bmatrix}
        \alpha\\
        \beta
    \end{bmatrix}\in E_T^\perp$. 
\end{proposition}
\begin{proof}
     We perform the next three steps: 
     
\noindent Step 1: If $\begin{bmatrix}
    \alpha\\
    \beta
\end{bmatrix}$ is supported outside all $U(2\varepsilon)$'s, the minimum of $g(X,X)$ is greater than $0$. Then, similar to \cite[Proposition 4.7]{wittendeformationweipingzhang}, we find 
\begin{align*}
    & \left\|\mathbb{D}_{T}\begin{bmatrix}
        \alpha\\
        \beta
    \end{bmatrix}\right\|\\
    \geqslant\ & \left\|\begin{bmatrix}
        & -d-d^*-T\hat{c}(X)\\
        d+d^*+T\hat{c}(X) & 
    \end{bmatrix}\begin{bmatrix}
        \alpha\\
        \beta
    \end{bmatrix}\right\|  - \left\|\begin{bmatrix}
        \dfrac{1}{2}(\omega^*-\omega) & \\
        & \dfrac{1}{2}(\omega-\omega^*)
    \end{bmatrix}\begin{bmatrix}
        \alpha\\
        \beta
    \end{bmatrix}\right\|\\
    \geqslant\ & C_7 T\left\|\begin{bmatrix}
        \alpha\\
        \beta
    \end{bmatrix}\right\|- C_8\left\|\begin{bmatrix}
        \alpha\\
        \beta
    \end{bmatrix}\right\|. 
\end{align*}

\noindent Step 2: If $\begin{bmatrix}
    \alpha\\
    \beta
\end{bmatrix}$ is supported inside the chart $U$ centered at some zero point $p$, we view $\begin{bmatrix}
    \alpha\\
    \beta
\end{bmatrix}$ as an element in $\Omega^\even(\mathbb{R}^{4n})\oplus\Omega^\odd(\mathbb{R}^{4n})$. Let $p_T'$ be the orthogonal projection from $\Omega^\even(\mathbb{R}^{4n})\oplus\Omega^\odd(\mathbb{R}^{4n})$ to the one-dimensional space generated by $\begin{bmatrix}
    \rho\\
    \eta
\end{bmatrix}$. Let $\langle\cdot,\cdot\rangle$ denote the inner product induced by (\ref{inner product on rn}), we have 
\begin{align*}
    p_T'\begin{bmatrix}
        \alpha\\
        \beta
    \end{bmatrix} =\ & \dfrac{1}{\sqrt{\|\rho\|^2+\|\eta\|^2}}\left\langle\begin{bmatrix}
        \rho\\
        \eta
    \end{bmatrix},\begin{bmatrix}
        \alpha\\
        \beta
    \end{bmatrix}\right\rangle\cdot\left(\dfrac{1}{\sqrt{\|\rho\|^2+\|\eta\|^2}}\begin{bmatrix}
        \rho\\
        \eta
    \end{bmatrix}\right)\\
    =\ & \dfrac{1}{\|\rho\|^2+\|\eta\|^2}\left\langle\begin{bmatrix}
        \rho\\
        \eta
    \end{bmatrix},\begin{bmatrix}
        \alpha\\
        \beta
    \end{bmatrix}\right\rangle\cdot\begin{bmatrix}
        \rho\\
        \eta
    \end{bmatrix}\\
    =\ & \dfrac{1}{\|\rho\|^2+\|\eta\|^2}\int_M (1-\gamma)\cdot g\left(\begin{bmatrix}
        \rho\\
        \eta
    \end{bmatrix},\begin{bmatrix}
        \alpha\\
        \beta
    \end{bmatrix}\right)\dvol\cdot\begin{bmatrix}
        \rho\\
        \eta
    \end{bmatrix} \ \ \\
    & \left(\text{Since}\ \left\langle\begin{bmatrix}
        \gamma\rho\\
        \gamma\eta
    \end{bmatrix},\begin{bmatrix}
        \alpha\\
        \beta
    \end{bmatrix}\right\rangle = 0\right). 
\end{align*}
Then, we find 
\begin{align*}
    & \left\|p_T'\begin{bmatrix}
        \alpha\\
        \beta
    \end{bmatrix}\right\| \\ =\  & \dfrac{1}{\sqrt{\|\rho\|^2+\|\eta\|^2}}\left|\int_M (1-\gamma)\cdot g\left(\begin{bmatrix}
        \rho\\
        \eta
    \end{bmatrix},\begin{bmatrix}
        \alpha\\
        \beta
    \end{bmatrix}\right)\dvol\right|\\
    \leqslant\ & \dfrac{1}{\sqrt{\|\rho\|^2+\|\eta\|^2}} \cdot\exp\left(-C_9\varepsilon^2 T\right)\int_{|{\mathbf{x}}|\leqslant 4\varepsilon} g\left(\begin{bmatrix}
        \alpha\\
        \beta
    \end{bmatrix},\begin{bmatrix}
        \alpha\\
        \beta
    \end{bmatrix}\right)^{1/2}\dvol \ \ \text{(By Cauchy-Schwarz)}\\
    \leqslant\ & \dfrac{\sqrt{T}}{\sqrt{T+C_1^2}}\cdot\|\rho\|^{-1}\cdot\exp\left(-C_9\varepsilon^2 T\right)\cdot\left\|\begin{bmatrix}
        \alpha\\
        \beta
    \end{bmatrix}\right\| \\
    \leqslant\ & \exp(-C_{10}T)\cdot\left\|\begin{bmatrix}
        \alpha\\
        \beta
    \end{bmatrix}\right\|\ \ \left(\text{By comparing $\|\rho\|$ with $\exp\left(-C_9\varepsilon^2 T\right)$}\right). 
\end{align*}
By Proposition \ref{higher part estimate}, we find 
\begin{align*}
    \left\|\mathbb{D}_{T}\begin{bmatrix}
        \alpha\\
        \beta
    \end{bmatrix}\right\| \geqslant\  & C_5\sqrt{T}\left\|\begin{bmatrix}
        \alpha\\
        \beta
    \end{bmatrix}-p_T'\begin{bmatrix}
        \alpha\\
        \beta
    \end{bmatrix}\right\|
    \geqslant C_5\sqrt{T}\cdot\left(1-\exp\left(-C_{10}T\right)\right)\cdot\left\|\begin{bmatrix}
        \alpha\\
        \beta
    \end{bmatrix}\right\|. 
\end{align*}

\noindent Step 3: The general $\begin{bmatrix}
    \alpha\\
    \beta
\end{bmatrix}\in E_T^\perp$ supported on $M$: We combine what we have verified in Step 1 and Step 2, following the standard procedure in the Step 3 of the proof of \cite[Proposition 4.12]{wittendeformationweipingzhang}. 
\end{proof}

Notice that $\mathbb{D}_T$ is skew-adjoint, we have: 
\begin{proposition}\label{spectrum result}
    The operator $-\mathbb{D}_{T}^2$ is self-adjoint and nonnegative. When $T$ is sufficiently large, the eigenvalues of $-\mathbb{D}_T^2$ lie in the union $[0, C_6^2]\cup [C_{11}^2 T, +\infty)$. 
\end{proposition}
\begin{proof}
   This is a combination of Proposition \ref{spectral left} and Proposition \ref{spectral right}, following the same pattern as in the proof of \cite[Lemma 5.3]{weipingzhangnewedition}. Since there is no essential spectrum here, we only need a simplified procedure like in the proof of \cite[Proposition 6.18]{zhuang2025analytictopologicalrealizationsinvariant}. 
\end{proof}

\section{Counting formula}\label{Counting formula section}
In this section, we prove the counting formula (\ref{main result main formula}) in Theorem \ref{counting formula theorem}. 
Let $\widetilde{E}_T$ be the sum of eigenspaces of $-\mathbb{D}_{T}^2$ on $\Omega^\even(M)\oplus\Omega^\odd(M)$ associated with eigenvalues in $[0,C_6^2]$. Then,
\begin{align*}
\kappa(M,\omega)
    =\ & \ind_2\left(\mathbb{D}_{T} \text{\ on\ }\Omega^\even(M)\oplus\Omega^\odd(M)\right) \\
    =\ & \hspace{-1mm}\dim\ker\left(-\mathbb{D}_{T}^2 \text{\ on\ }\Omega^\even(M)\oplus\Omega^\odd(M)\right) \hspace{-2.5mm}\mod 2\\
    =\ & \hspace{-1mm}\dim\ker\left(\mathbb{D}_{T}: \widetilde{E}_T\to\widetilde{E}_T\right) \hspace{-2.5mm}\mod 2 \\
    &\text{(Since each eigenspace of $-\mathbb{D}_T^2$ is invariant under $\mathbb{D}_T$).}
\end{align*}
By \cite[Section 8.16]{greub1981linear}, every $r\times r$ skew-symmetric matrix has Atiyah-Singer mod 2 index equal to the parity of $r$. Thus, $$\kappa(M,\omega) = \dim\widetilde{E}_T \mod 2.$$
Now, to prove Theorem \ref{counting formula theorem}, we only need to show that $\dim E_T = \dim\widetilde{E}_T$. 
\begin{proposition}\label{checking same dimension of kernel spaces}
    We have $\dim E_T = \dim\widetilde{E}_T$ when $T$ is sufficiently large.
\end{proposition}
\begin{proof}
    Recall the $L^2$-norm (\ref{L2 on pair}) on $\Omega^\even(M)\oplus\Omega^\odd(M)$. We let $$\widetilde{P}_T:\Omega^\even(M)\oplus\Omega^\odd(M)\to\widetilde{E}_T$$
    be the orthogonal projection to $\widetilde{E}_T.$ Then, for any $h\in E_T$, we obtain
    \begin{align*}
         \|h-\widetilde{P}_Th\|
        \leqslant\ & \dfrac{1}{C_{11}\sqrt{T}}\|\mathbb{D}_T(h-\widetilde{P}_Th)\|\ \ (\text{by Proposition \ref{spectrum result}})\\
        \leqslant\ & \dfrac{1}{C_{11}\sqrt{T}}\left(\|\mathbb{D}_Th\|+\|\mathbb{D}_T\widetilde{P}_Th\|\right)\\
        \leqslant\ & \dfrac{1}{C_{11}\sqrt{T}}\cdot C_6\cdot (\|h\|+\|h\|)\ \ \text{(by Proposition \ref{spectrum result})}.
    \end{align*}
    Thus, when $T$ is large, $\widetilde{P}_T$ maps $E_T$ injectively into $\widetilde{E}_T$, meaning that $\dim\widetilde{E}_T\geqslant\dim E_T$. 

\vspace{+1mm}
    Now, similar to \cite[(5.32)]{wittendeformationweipingzhang}, suppose that $\dim\widetilde{E}_T > \dim E_T$, we pick some $\varphi\in \widetilde{E}_T$ such that $\varphi$ is orthogonal to the space $\widetilde{P}_TE_T$. Let $\langle\cdot,\cdot\rangle$ denote the inner product induced by (\ref{L2 on pair}), for any zero point $p$ of $X$ and the associated $\begin{bmatrix}
        \rho_p\\
        \eta_p
    \end{bmatrix}$ (or $\begin{bmatrix}
        \eta_p\\
        \rho_p
    \end{bmatrix}$ depending on the sign of $\det A$), we have 
    \begin{align*}
        \left\langle\varphi,\begin{bmatrix}
            \rho_p\\
            \eta_p
        \end{bmatrix}\right\rangle
        =\ & \left\langle\varphi,\begin{bmatrix}
            \rho_p\\
            \eta_p
        \end{bmatrix}\right\rangle - \left\langle\varphi,\widetilde{P}_T\begin{bmatrix}
            \rho_p\\
            \eta_p
        \end{bmatrix}\right\rangle\\
        =\ & \left\langle\varphi,\begin{bmatrix}
            \rho_p\\
            \eta_p
        \end{bmatrix}\right\rangle - \left\langle\varphi, \begin{bmatrix}
            \rho_p\\
            \eta_p
        \end{bmatrix}\right\rangle\ \ \text{(Since $\varphi\in\widetilde{E}_T$)}\\
        =\ & 0. 
    \end{align*}
    Thus, $\varphi\in E_T^\perp$. By Proposition \ref{spectral right}, 
    $$\|\mathbb{D}_{T}\varphi\|\geqslant C_{11}\sqrt{T}\|\varphi\|,$$
    contradicting to the fact that 
    \begin{align*}
        \varphi\in\widetilde{E}_T =\ & \text{the sum of the eigenspaces of $-\mathbb{D}_{T}^2$} \text{associated with eigenvalues in $[0,C_6^2]$}. 
    \end{align*}
    Therefore, $\widetilde{E}_T$ is isomorphic to $E_T$ when $T$ is sufficiently large. 
\end{proof}

Recall that $X$ is adjusted from $V$, by Proposition \ref{checking same dimension of kernel spaces}, we finally have
    \begin{align*}
    \kappa(M,\omega)
        =\ & \hspace{-1mm}\dim\widetilde{E}_T \mod 2\\
    =\ & \hspace{-1mm}\dim E_T \mod 2\\
    =\ & \text{the number of zero points of the adjusted vector field}\ X \mod 2\\
    =\ & \text{the number of zero points of the original vector field}\ V \mod 2\ ,
    \end{align*}
and this completes the proof of Theorem \ref{counting formula theorem}.

\begin{remark}\label{atiyah trick}
\normalfont
We get Corollary \ref{vanishing theorem} from Theorem \ref{counting formula theorem}. Actually, instead of using Theorem \ref{counting formula theorem}, we can also apply Atiyah's perturbation technology in \cite[Section 4]{atiyah2013vector} to prove Corollary \ref{vanishing theorem} directly.
    Let $V$ be a vector field with $g(V,V) = 1$ on $M$. Then, we perturb the operator
    $$D = \begin{bmatrix}
        0 & 1\\
        1 & 0
    \end{bmatrix}\begin{bmatrix}
        \hat{c}(\dvol) & \\
        & \hat{c}(\dvol)
    \end{bmatrix}\begin{bmatrix}
        d+d^* & \omega\\
        \omega^* & -d-d^*
    \end{bmatrix}$$
    on $\Omega^\even(M)\oplus\Omega^\odd(M)$ into the operator $$D' = D + \begin{bmatrix}
        0 & 1\\
        1 & 0
    \end{bmatrix}\begin{bmatrix}
        \hat{c}(V) & \\
        & -\hat{c}(V)
    \end{bmatrix} D \begin{bmatrix}
        \hat{c}(V) & \\
        & -\hat{c}(V)
    \end{bmatrix}\begin{bmatrix}
        0 & 1\\
        1 & 0
    \end{bmatrix}.$$
    Once we verify that $\ind_2 D = \ind_2 D'$ and that $\ker D'$ admits a complex structure given by 
    $$\left(\begin{bmatrix}
        0 & 1\\
        1 & 0
    \end{bmatrix}\begin{bmatrix}
        \hat{c}(V) & \\
        & -\hat{c}(V)
    \end{bmatrix}\right)^2 = \begin{bmatrix}
        -1 & 0\\
        0 & -1
    \end{bmatrix},$$
    we conclude that $\dim\ker D'$ is even, and therefore $\kappa(M,\omega) = 0$. 
\end{remark}

\section{Examples and discussions}\label{Examples section}
In this section, we present some examples. Most of them have already been studied in other papers, and we just adapt them into the computation of symplectic semi-characteristics. After these examples, we propose some further discussions. 
\begin{example}\normalfont
    We let $M = \cptwo$ equipped with the Fubini-Study form \cite[Homework 12]{daSilva2008}. According to \cite[Example 4.2]{tangtsengclausensymplecticwitten}, 
    $$b_0^\omega = 1, b_2^\omega = 0, b_4^\omega = 0,$$
    meaning that $\kappa(M,\omega) = 1$. 
    These $b_i^\omega$\hspace{+0.5mm}'s are computed using
    a Morse function with $3$ critical points together with the associated cone Morse cochain complex \cite[Definition 1.2]{tangtsengclausensymplecticwitten}. By the counting formula (\ref{main result main formula}), we can also use these $3$ critical points of this perfect Morse function to find   $\kappa(M,\omega) = 1$. 
\end{example}

\begin{example}\normalfont \label{weak replacement of the euler char}
\normalfont
    Let $M = \mathbb{S}^2\times\mathbb{S}^2$ equipped with the standard symplectic structure. 
    Recall that we have a height function $h$ (see \cite[Example 3.4]{banyaga2013lectures}) on $\mathbb{S}^2$ with $2$ critical points. Then, 
    \begin{align*}
    f: \mathbb{S}^2\times\mathbb{S}^2&\to\mathbb{R}\\
    (p,q)&\mapsto h(p)+h(q)
    \end{align*}
     is a Morse function on $\mathbb{S}^2$ with $4$ critical points. Thus, in this case, 
    $\kappa(M,\omega) = 0$.

    As we know, the Euler characteristic of the de Rham cohomology of $\mathbb{S}^2\times\mathbb{S}^2$ is $4$, meaning that $\mathbb{S}^2\times\mathbb{S}^2$ does not admit a nonvanishing vector field. However, as we see, its symplectic semi-characteristic is $0$. Thus, in terms of judging the existence of nonvanishing vector fields, the symplectic semi-characteristic (\ref{symplectic semi definition we need it}) of the primitive cohomology is a weak substitute of the Euler characteristic of the de Rham cohomology. 
\end{example}

\begin{example}\normalfont
    According to \cite[Section 3.4]{tty1st} and \cite[(5.3)]{tanaka_tseng_2018}, we let $\sim$ be the identification 
   $$(x_1,x_2,x_3,x_4)\sim (x_1 + a, x_2+b, x_3+c, x_4+d - bx_3) \text{\ \ (when $a,b,c,d\in\mathbb{Z}$)}$$
    on $\mathbb{R}^4$. Then, the Kodaira-Thurston four-fold is equal to $\mathbb{R}^4/\sim$. Let $M = \mathbb{R}^4/\sim$ equipped with the symplectic form $\omega$ given in \cite[(3.26)]{tty1st} and \cite[(5.4)]{tanaka_tseng_2018}. On the one hand, we can immediately see $\kappa(M,\omega) = 0$ from the tables of primitive cohomology groups provided in \cite[Section 3.4]{tty1st} and \cite[Section 5.4]{tanaka_tseng_2018}. On the other hand, since $\mathbb{R}^4/\sim$ has a globally defined tangent vector field $\partial_{x_1}$, we also obtain $\kappa(M,\omega) = 0$ according to Corollary \ref{vanishing theorem}. 
\end{example}

\begin{example}\label{counter example 2}
\normalfont
 We mention a little bit about the $(4n+2)$-dimensional case. For example, we let $M = \mathbb{T}^2$ equipped with the standard symplectic form. Since $\mathbb{T}^2$ is K\"ahler, we use the formula \cite[(4.4)]{tangtsengclausensymplecticwitten} to find 
 $b_0^\omega = 1, b_2^\omega = 2$, 
 and then $\kappa(M,\omega) = 1$. 
 
 However, we know there is a height function \cite[Part I, Section 1]{milnor1963morse} with $4$ nondegenerate critical points on $\mathbb{T}^2$. This means our Theorem \ref{counting formula theorem} does not apply to the $(4n+2)$-dimensional case. 
   The same thing happens to $\mathbb{S}^2$ on which we have the height function \cite[Example 3.4]{banyaga2013lectures}. Thus, we still need more subtle investigation into the $(4n+2)$-dimensional case.
\end{example}

We now propose some discussions for the extensions of this project. The deformation replaces $d+d^*$ by $d+d^*+T\hat{c}(V)$ and does not perturb $\omega$. This preserves the symplectic information and relates $\kappa(M,\omega)$ to the primitive forms given by the Lefschetz decomposition. However, as $\kappa(M,\omega)$ is unchanged when replacing $\omega$ by another symplectic form, it is also natural to consider replacing $\omega$ by $\omega\wedge\cdots\wedge\omega$ or by other forms. 

The $\omega\wedge\cdots\wedge\omega$ gives the semi-characteristic of the $1$-filtered cohomology \cite{tanaka_tseng_2018, tty3rd}. In a follow-up work \cite{follow_up_1}, we use $\omega\wedge\omega$ on a $(4n+2)$-dimensional closed symplectic manifold. Then, the associated semi-characteristic vanishes, which is exactly the parity of the de Rham Euler characteristic. This provides more evidence that $\kappa$ relies more on $M$ than on the form. 

Also, from an index-theoretic perspective without involving the primitive cohomology, we may perturb $\omega$ into $s\omega$ and let $s$ change from $0$ to $1$. The study could thus be done for any closed orientable manifold equipped with a closed homogeneous form. For this case, we will give the formula and the analysis in a follow-up work \cite{follow_up_2} with assumptions on both $\dim M$ and the degree of the form. A similar perturbation using $s^{-1}$ instead of $s$ are carried out in a recent work \cite{morse_recent_mapping_cone} on the mapping cone Morse theory for any closed oriented manifold equipped with a closed homogeneous smooth form. This $s^{-1}$ preserves the cohomology. 

Finally, we briefly discuss the $K$-theoretic background of this study. By \cite[Theorem 2.3]{atiyah_singer_1971index_V}, the mod 2 index is equivalent to the map 
\begin{align}\label{index map mod 2}
    KO^{-1}(TM)\to KO^{-1}(\text{point}) \cong \mathbb{Z}_2.
\end{align}
When $\dim M = 4n$, we have a skew-adjoint elliptic operator (\ref{the witten deformation of the skew adjoint operator constructed above and use later}) whose skew-symbol class is in $KO^{-1}(TM)$ and then mapped to $\kappa(M,\omega)$. In the $(4n+2)$-dimensional case, if we continue the current pattern of construction, we cannot obtain a skew-adjoint elliptic operator on $\Omega^\even(M)\oplus\Omega^\odd(M)$ whose skew-symbol class is mapped to $\kappa(M,\omega)$. Thus, in the $4n+2$ case, it could be worthwhile to study the geometric meaning of the image of the skew-symbol class of (\ref{before witten def}) under (\ref{index map mod 2}) (cf. \cite[Theorem 3.2]{zhangcountingmod2indexkervairesemi}), and whether (\ref{index map mod 2}) can be generalized to give the $(4n+2)$-dimensional case a $KO$-valued index result of the semi-characteristic. 

\bibliographystyle{abbrv}
\bibliography{mybib.bib}
\end{document}